\documentclass{article}
\usepackage{amsmath, amssymb, amsthm}
\usepackage[dvips]{graphicx}
\newtheorem{dfn}{Definition}[section]
\newtheorem{prop}[dfn]{Proposition}
\newtheorem{thm}[dfn]{Theorem}
\newtheorem{rem}[dfn]{Remark}
\newcommand{\R}{\mathbb{R}}
\newcommand{\rsp}{\mathcal{R}}
\newcommand{\rsa}{\mathcal{R}^\ast}
\newcommand{\mat}{\mathcal{M}}
\newcommand{\chix}[1]{\chi \! \genfrac{}{}{0pt}{1}{}{#1}}
\DeclareMathOperator{\sgn}{sgn}
\title{New examples of oriented matroids with disconnected realization spaces}
\author{Yasuyuki Tsukamoto
\thanks{Department of Mathematics, Kyoto University,
Sakyo-ku, Kyoto 606-8502, Japan
tsukam@math.kyoto-u.ac.jp}}
\begin{document}
\maketitle
\begin{abstract}
We construct oriented matroids of rank 3 on 13 points
whose realization spaces are disconnected.
They are defined on smaller points
than the known examples with this property.
Moreover, we construct the one on 13 points
whose realization space is a connected
and non-irreducible semialgebraic variety.
\end{abstract}
\section{Oriented Matroids and Matrices}
Throughout this section, we fix positive integers $r$ and $n$. 

Let $X=(x_1,\dots ,x_n)\in \R^{rn}$
be a real $(r,n)$ matrix of rank $r$,
and $E=\{ 1,\dots ,n\}$ be a set of labels of the columns of $X$.
For such matrix $X$, a map $\chix{X}$ can be defined as
\[ \chix{X}:E^r\rightarrow \{-1,0,+1\},\quad
\chix{X}(i_1, \dots ,i_r):=\sgn \det (x_{i_1},\dots ,x_{i_r}).\]
The map $\chix{X}$ is called the {\it chirotope} of $X$.
The chirotope $\chix{X}$ encodes the information
on the combinatorial type which is called the
{\it oriented matroid} of $X$.
In this case, the oriented matroid
determined by $\chix{X}$ is of rank $r$ on $E$.

We note for some properties
which the chirotope $\chix{X}$ of a matrix $X$ satisfies.
\begin{enumerate}
\item $\chix{X}$ is not identically zero.
\item $\chix{X}$ is alternating, i.e.
$\chix{X}(i_{\sigma(1)},\dots,i_{\sigma(r)})
=\sgn (\sigma) \chix{X}(i_1,\dots,i_r)$\\
for all $i_1,\dots ,i_r\in E$ and all permutation $\sigma$.
\item For all $i_1,\dots ,i_r,j_1,\dots ,j_r\in E$ such that\\
$\chix{X} (j_k,i_2,\dots ,i_r)\cdot
\chix{X} (j_1,\dots ,j_{k-1},i_1,j_{k+1},\dots, j_r)\geq 0$
for $k=1,\dots ,r$,\\
we have $\chix{X}(i_1,\dots,i_r)\cdot
\chix{X}(j_1,\dots,j_r) \geq 0$.
\end{enumerate}
The third property follows from the identity
\begin{multline*}
\det (x_1,\dots, x_r)\cdot \det (y_1,\dots ,y_r)\\
=\sum_{k=1}^r \det (y_k,x_2,\dots ,x_r)\cdot
\det (y_1,\dots ,y_{k-1},x_1,y_{k+1},\dots, y_r),\\
\text{for all } x_1,\dots,x_r,y_1,\dots,y_r\in \R^r.
\end{multline*}

Generally, an oriented matroid of rank $r$ on $E$ ($n$ points)
is defined by a map $\chi:E^r\rightarrow \{-1,0,+1\}$,
which satisfies the above three properties~(\cite{bjo}).
The map $\chi$ is also called the chirotope
of an oriented matroid.
We use the notation $\mat (E,\chi )$ for an oriented matroid
which is on the set $E$ and is defined by the chirotope $\chi$.

An oriented matroid  $\mat (E, \chi )$ is
called \textit{realizable} or \textit{constructible},
if there exists a matrix $X$ such that
$\chi=\chix{X}$.
Not all oriented matroids are realizable,
but we don't consider non-realizable case in this paper.

\begin{dfn}%%%%%%%%%%        Realization     %%%%%%%%
A realization of an oriented matroid
$\mat=\mat (E, \chi )$ is a matrix
$X$ such that $\chix{X}=\chi$ or $\chix{X}=-\chi$.
\end{dfn}
Two realizations $X,X^\prime$ of $\mat$ is called
linearly equivalent,
if there exists a linear transformation $A\in GL(r,\R)$ such that
$X^\prime =AX$.
Here we have the equation
$\chix{X^\prime}=\sgn (\det A)\cdot \chix{X}$.

\begin{dfn}%%%%%%      Realization Space    %%%%%%%%
The realization space $\rsp(\mat)$ of an oriented matroid $\mat$
%of rank $r$ on $n$ points 
is the set of all linearly equivalent classes of realizations of $\mat$,
in the quotient topology induced from $\R^{rn}$.
\end{dfn}

%%%%%%%      History   of   the   Isotopy  Conjecture   %%%%%%%
Our motivation is as follows:
In 1956, Ringel asked whether
the realization spaces $\rsp(\mat)$ are 
necessarily connected~\cite{rin}.
It is known that every oriented matroid 
on less than $9$ points has a contractible
realization space.
In 1988, Mn\"{e}v showed that $\rsp(\mat)$ 
can be homotopy equivalent
to an arbitrary semialgebraic variety~\cite{mne}.
His result implies that they can have arbitrary 
complicated topological types.
In particular, there exist oriented matroids with 
disconnected realization spaces.
Suvorov and Righter-Gebert constructed such examples 
of oriented matroids of rank $3$ on $14$ points,
in 1988 and in 1996 respectively~\cite{suv,ric}.
However it is unknown which is the smallest number of points
on which oriented matroids can have
disconnected realization spaces.
See \cite{bjo} for more historical comments.

One of the main results of this paper is the following.
\begin{thm}\label{mt1}
There exist oriented matroids of rank $3$ on $13$ points
whose realization spaces are disconnected.
\end{thm}
Realization spaces of oriented matroids are semialgebraic varieties.
So it makes sense whether a realization space is irreducible or not.
As far as the author knows, no example of an oriented matroid
with a connected and non-irreducible realization space
has been explicitly given.
\begin{thm}
There exists oriented matroid of rank $3$ on $13$ points
whose realization space is connected and non-irreducible.
\label{mt2}
\end{thm}

\vspace{0.5\baselineskip}
\noindent
\textbf{Acknowledgment.} I would like to thank Masahiko Yoshinaga
for valuable discussions and comments.
I also thank Yukiko Konishi for comments on the manuscript.

\section{Construction of the examples}
Throughout this section, we set $E=\{1,\dots, 13\}$.

We will define three chirotopes $\chi^- ,\chi^0$ and $\chi^+$.
We will see that $\mat(E,\chi^-)$ and $\mat(E,\chi^+)$
have disconnected realization spaces,
and $\mat(E,\chi^0)$ has
a connected and non-irreducible realization space.

Let $X(s,t,u)$ be a real $(3,13)$ matrix
with three parameters $s,t,u \in \R$ given by
\begin{multline*}
X(s,t,u):=(x_1,\dots, x_{13})\\
=\left(
\begin{array}{ccccccccccc}
1&0&0&1&s&s&0&1&1&st&s+t-u-st+su\\
0&1&0&1&0&1&t&t&u&t&t-u+su\\
0&0&1&1&1&1&1&1&0&1-su&1-u+su
\end{array}
\right.\\
\left.
\begin{array}{cc}
s+t-st-s^2u&s(t-u+su)\\
t&t-u+su\\
1-su&1-u+su
\end{array}
\right) .
\end{multline*}

This is a consequence of the computation
of the following construction sequence.
Both operations ``$\vee$'' and ``$\wedge$'' can be computed
in terms of  the standard cross product ``$\times$'' in $\R^3$.
The whole construction
depends only on the choice of the three parameters $s,t,u\in\R$.
\begin{align*}
x_1&=\mathstrut^t(1,0,0),\,
x_2=\mathstrut^t(0,1,0),\,
x_3=\mathstrut^t(0,0,1),\,
x_4=\mathstrut^t(1,1,1),\\
x_5&=s\cdot x_1+x_3,\\
x_6&=(x_1\vee x_4)\wedge (x_2 \vee x_5),\\
x_7&=t\cdot x_2+x_3,\\
x_8&=(x_1 \vee x_7)\wedge (x_2\vee x_4),\\
x_9&=u\cdot x_2+x_1,\\
x_{10}&=(x_7\vee x_9) \wedge (x_3\vee x_6),\\
x_{11}&=(x_4\vee x_5) \wedge (x_8\vee x_9),\\
x_{12}&=(x_1 \vee x_{10}) \wedge (x_4\vee x_5),\\
x_{13}&=(x_3\vee x_6)\wedge (x_1 \vee x_{11}).
\end{align*}

We set $X_0=X\left( \frac{1}{2},\frac{1}{2},\frac{1}{3}\right)$.
The chirotope $\chi^\epsilon$ is the alternating map such that
\begin{multline*}
\chi^\epsilon (i,j,k) =\left\{
\begin{array}{lc}
\epsilon & \text{if } (i,j,k)=(9,12,13),\\
\chix{X_0}(i,j,k)& \text{otherwise,}
\end{array}
\right.\\
\text{for all }(i,j,k)\in E^3(i<j<k),%\label{chie}
\end{multline*}
where $\epsilon \in \{-,0,+\}$.

The oriented matroid which we will study is
$\mat^\epsilon :=\mat(E,\chi^\epsilon)$.

\begin{rem}
\rm We can replace $X_0$ with
$X\left(\frac{1}{2},\frac{1}{2},u^\prime \right)$
where $u^\prime$ is chosen from
$\R\backslash\{-1,0,\frac{1}{2},1,\frac{3}{2},2,3\}$.
We will study the case $0<u^\prime<\frac{1}{2}$.
If we choose $u^\prime$ otherwise,
we can get other oriented matroids
with disconnected realization spaces.
\end{rem}

In the construction sequence,
we need no assumption on the collinearity of $x_9,x_{12},x_{13}$.
Hence every realization of $\mat^\epsilon$
is linearly equivalent to a matrix $X(s,t,u)$
for certain $s,t,u$,
up to multiplication on each column with positive scalar.

Moreover, we have the rational isomorphism
\[
\rsa (\chi^\epsilon )\times (0,\infty )^{12}
\cong \rsp (\mat^\epsilon ), 
\]
where $\rsa (\chi^\epsilon ):=
\{ X(s,t,u)\in \R^{3\cdot 13}| \,\chix{X(s,t,u)}
=\chi^\epsilon \}$.
Thus we have only to prove that
the set $\rsa(\chi^\epsilon)$
is disconnected (resp. non-irreducible)
to show that the realization space $\rsp(\mat^\epsilon)$
is disconnected
(resp. non-irreducible).

The equation $\chix{X(s,t,u)}=\chi^\epsilon$ means that
\begin{equation}
\sgn \det(x_i,x_j,x_k)=\chi^\epsilon (i,j,k),
\text{ for all }(i,j,k)\in E^3.\label{ch}
\end{equation}
We write some of them which give the equations
on the parameters $s,t,u$.
Note that
for all $(i,j,k)\in E^3(\{i,j,k\}\neq \{9,12,13\})$,
the sign is given by
\[\chi^\epsilon (i,j,k)=
\sgn \det (x_i,x_j,x_k)|_{s=t=1/2,u=1/3}.\]
From the equation
$\sgn \det (x_2,x_3,x_5)=\sgn (s) =\sgn (1/2)=+1,$
we get $s>0$.
Similarly, we get $\det (x_2,x_5,x_4)=1-s>0$,
therefore
\begin{equation}
0<s<1.\label{s}
\end{equation}
From the equations
$\det(x_1,x_7,x_3)=t>0,\det(x_1,x_4,x_7)=1-t>0$,
we get
\begin{equation}
0<t<1.\label{t}
\end{equation}
Moreover, we have the inequalities
\begin{align}%%%%%    Inequalities  on   U  %%%%%%
\det (x_1, x_9, x_3)&=u>0,\label{u}\\
\det (x_4, x_7, x_9)&=1-t-u>0,\label{a}\\
\det (x_3,x_9,x_8)&=t-u>0,\label{b}\\
\det (x_5,x_{13},x_7)&=s\bigl( t^2-(1-s)u\bigr) >0,\label{c}\\
\det (x_6, x_{12},x_8)&=(1-s)\bigl( (1-t)^2-su \bigr) >0.\label{d}
\end{align}
From the equation $\det (x_9,x_{12},x_{13})=u(1-2s)(1-2t+tu-su)$,
we get
\begin{equation}
\sgn \bigl(u(1-2s)(1-2t+tu-su)\bigr)=\epsilon. \label{e}
\end{equation}
Conversely,
if we have Eqs. (\ref{s}) - (\ref{e}),
then we get (\ref{ch}).

We can interpret a $(3,13)$ matrix
as the set of vectors $\{x_1,\dots ,x_{13}\}\subset\R^3$.
After we normalize the last coordinate for
$x_i\, (i\in E\backslash \{ 1,2,9\} )$,
we can visualize the matrix
on the affine plane $\{ (x,y,1)\in \R^3\} \cong \R^2$.
Figure \ref{chioq} shows the affine image of $X_0$.
See Figures \ref{chipm}, \ref{chio}
for realizations of $\mat^\epsilon$.
\begin{figure}%%%%%%%           Figure  X  0          %%%%%%%%%
\centering
\includegraphics[width=6cm]{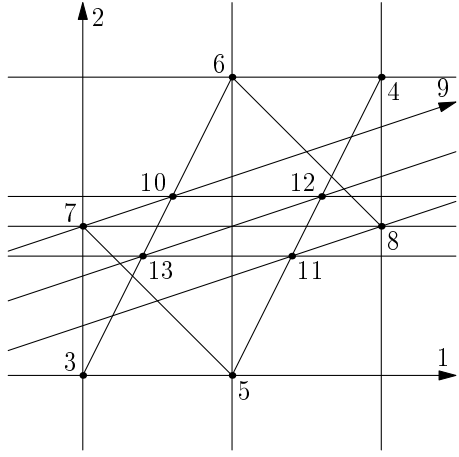}
\caption{Column vectors of $X_0$.}
\label{chioq}
\end{figure}
\begin{figure}%%%%% Realizations  of  M  +-0 Figure  %%%%
\centering
\includegraphics[width=5cm]{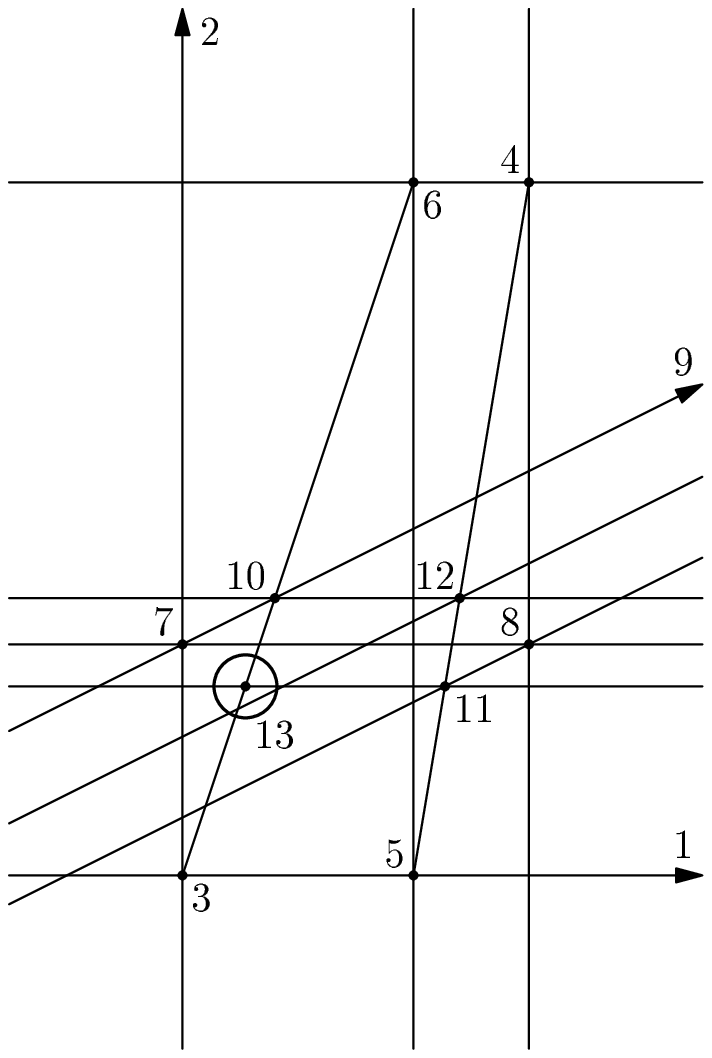}
\qquad
\includegraphics[width=5cm]{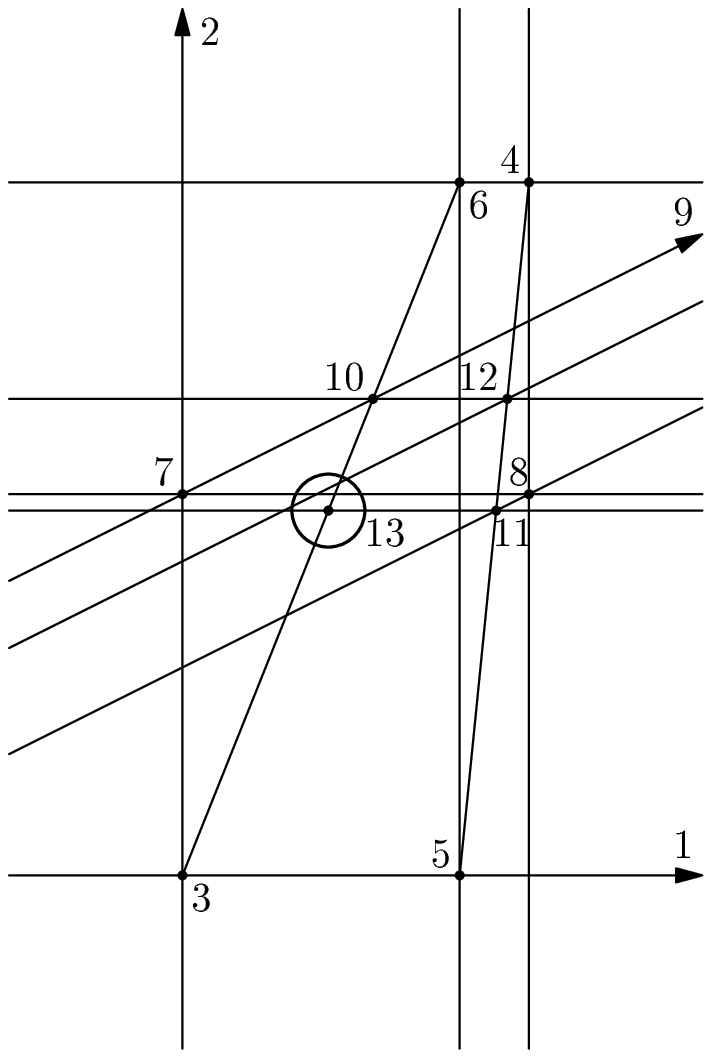}
\caption{Realization of $\mat^-$(on the left)
and that of $\mat^+$(on the right).}
\label{chipm}
\end{figure}
\begin{figure}%%%  Realizations of M 0 Figure     %%%%%% 
\centering
\includegraphics[width=5cm]{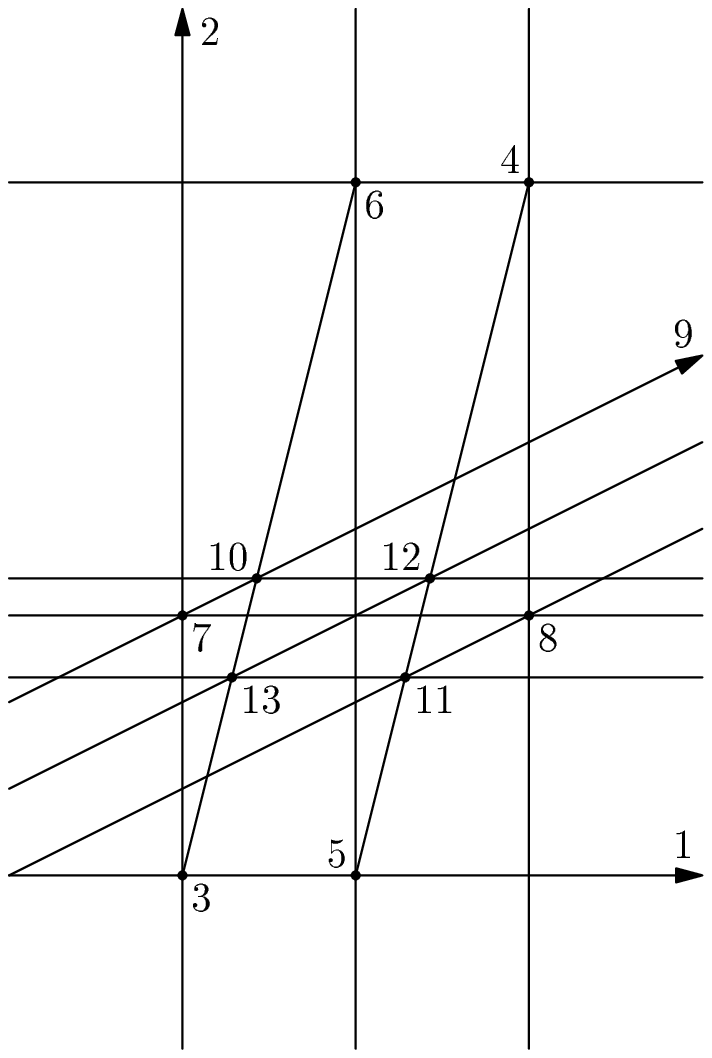}
\qquad
\includegraphics[width=5cm]{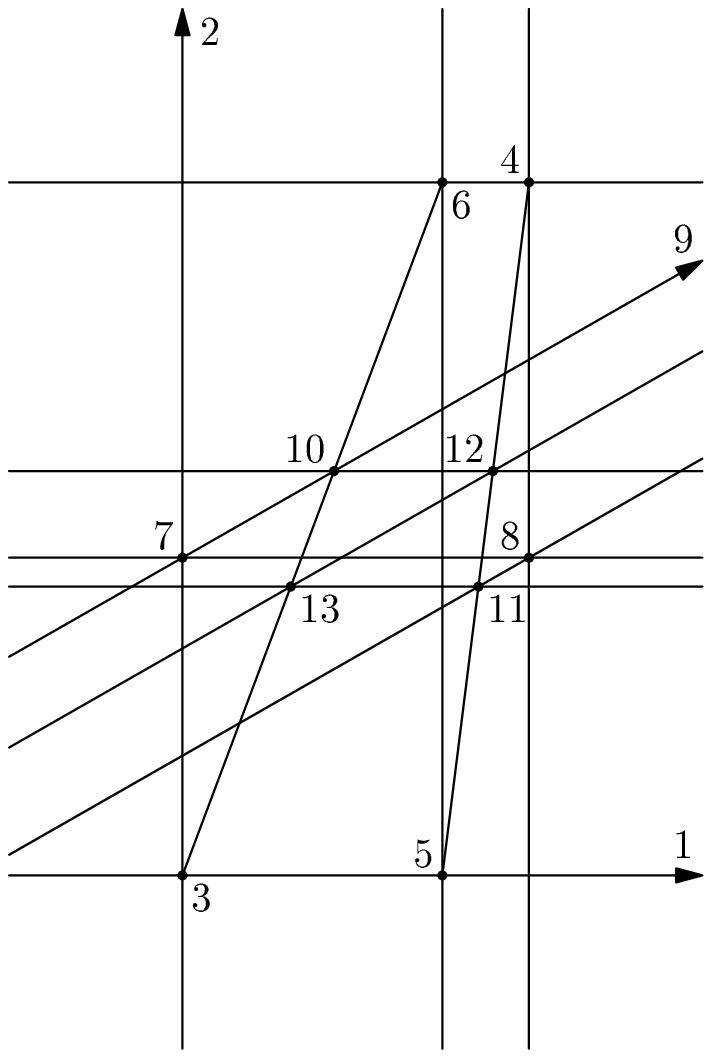}
\caption{Realizations of $\mat^0$.}
\label{chio}
\end{figure}

\vspace{0.7\baselineskip}
\noindent
{\it Proof of Theorem \ref{mt1}}.
We prove that $\rsa(\chi^-)$ and $\rsa(\chi^+)$ are disconnected.
From Eqs. (\ref{s}) - (\ref{e}), we obtain
\[
\rsa (\chi^- )\cong \left\{
(s,t,u)\in \R^3 \,\,
\begin{array}{|c}
0<s<1,\,0<u<t<1-u, \\
(1-t)^2-su>0,\,t^2-(1-s)u>0, \\
(1-2s)(1-2t+tu-su)<0
\end{array}
\right\},
\]
\[
\rsa (\chi^+ )\cong \left\{
(s,t,u)\in \R^3 \,\,
\begin{array}{|c}
0<s<1,\,0<u<t<1-u,\\
(1-t)^2-su>0,\,t^2-(1-s)u>0,\\
(1-2s)(1-2t+tu-su)>0
\end{array}
\right\} .
\]

First, we show that $\rsa (\chi^- )$ is disconnected,
more precisely, consisting of two connected components,
by proving the next proposition.

\begin{prop}%%%%%   Prop      Real  (   Chi   ----  )   %%%%%%
\[
\begin{split}
\rsa (\chi^- )\cong&
\left\{ (s,t,u)\! \in \R^3 \,
\begin{array}{|l}
0\! <\! s\! <\! 1/2\\
1/2\! <\! t\! <\! 1
\end{array}
,0\! <\! u\! <\! \min \left\{1\! -\! t,
\frac{(1\!-\!t)^2}{s},\,\frac{2t\!-\!1}{t\!-\!s} \right\} 
\right\} \\
&\cup \left\{ (s,t,u)\! \in \R^3 \,
\begin{array}{|l}
1/2<s<1\\
0<t<1/2
\end{array}
,\,0<u<\min \left\{ t,\, \frac{t^2}{1-s},\,
\frac{1-2t}{s-t}\right\}
\right\} .
\end{split}
\]
\label{prom}
\end{prop}
\begin{proof}%%%%%%%%     Proof   of  disconnectedness   %%%%%%%
There are two cases
\begin{equation*}
(1-2s)(1-2t+tu-su)<0
\Leftrightarrow
\left\{
\begin{array}{c}
1-2s>0,\,1-2t+tu-su<0,\\
\text{or} \\
1-2s<0,\, 1-2t+tu-su>0.
\end{array}
\right. %\label{case}
\end{equation*}
Note that
\begin{align}
&(2-u)(2t-1)=-2(1-2t+tu-su)+u(1-2s),\label{thalf}\\
&t^2-(1-s)u=-(1-2t+tu-su)+(1-t)(1-t-u),\label{tsq}\\
&(1-t)^2-su=(1-2t+tu-su)+t(t-u).\label{1mtsq}
\end{align}
($\subset$)
For the case $1-2s>0$ and $1-2t+tu-su<0$,
the inequality $2t-1>0$ follows from Eq. (\ref{thalf}).
Since we have $0<s<1/2<t<1$, we get
\begin{equation}
\left\{
\begin{array}{l}
1-2t+tu-su<0,\\(1-t)^2-su>0,\\1-t-u>0
\end{array}\right.
\Leftrightarrow u<\min\left\{1-t,\,
\frac{(1-t)^2}{s},\, \frac{2t-1}{t-s}\right\}.
\label{ssmall}
\end{equation}

For the other case $1-2s<0$, similarly,
we get $1-2t>0$ from Eq. (\ref{thalf}).
Since we have $0<t<1/2<s<1$, we get
\begin{equation}
\left\{
\begin{array}{l}
1-2t+tu-su>0,\\ t^2-(1-s)u>0,\\t-u>0
\end{array}\right.
\Leftrightarrow u<\min
\left\{ t,\,\frac{t^2}{1-s},\,\frac{1-2t}{s-t} \right\}.
\label{tsmall}
\end{equation}
($\supset$)
For the component $0<s<1/2<t<1$,
the inequalities $1-2t+tu-su<0,\,(1-t)^2-su>0,\,1-t-u>0$
follow from (\ref{ssmall}).
Thus we get $t^2-(1-s)u>0$ from Eq. (\ref{tsq}).
The inequality $u<t$ holds because $t>1/2$ and $u<1-t$.

For the other component $0<t<1/2<s<1$, similarly,
we get the inequalities
$1-2t+tu-su>0,\, t^2-(1-s)u>0,\,t-u>0$ from (\ref{tsmall}),
and $(1-t)^2-su>0$ from Eq. (\ref{1mtsq}).
Last, we get $u<1-t$ from $t<1/2$ and $u<t$.
\end{proof}

For the set $\rsa (\chi^+ )$,
we have the following proposition.
\begin{prop}%%%%%%%%     Prop    Real  (  Chi  ++++ )   %%%%%
\begin{align*}
\rsa(\chi^+)\cong&
\left\{ (s,t,u)\in \R^3 \left|
\begin{array}{c}
0<s<1/2,0<u<1/2,\\
(1-u)^2-(1-s)u>0,
\end{array}
\!\sqrt{(1\!-\!s)u}<t<\frac{1\!-su}{2-u}
\right. \right\} \notag \\
\cup&
\left\{ (s,t,u)\in \R^3 \left|
\begin{array}{c}
1/2<s<1,0<u<1/2,\\
(1-u)^2-su>0,\end{array}
\frac{1-su}{2-u}<t<1-\sqrt{su}
\right. \right\} .\label{chip}
\end{align*}

\end{prop}
The proof is similar to that of Proposition \ref{prom}
and omitted.

\vspace{0.7\baselineskip}
\noindent 
{\it Proof of Theorem \ref{mt2}.}
We show that
$\rsa (\chi^0)$ consists of two irreducible components
whose intersection is not empty.
From Eqs. (\ref{s}) - (\ref{e}), we get
\[
\rsa (\chi^0 )
\cong \left\{ (s,t,u)\in \R^3\,
\begin{array}{|c}
0<s<1,\,0<u<t<1-u, \\
(1-t)^2-su>0,\,t^2-(1-s)u>0,\\
(1-2s)(1-2t+tu-su)=0
\end{array}
\right\}.
\]
Here we have the decomposition
\[
\begin{split}
\rsa (\chi^0)\cong &\Bigl\{ (s,t,u)\in \R^3 \Bigm|
0<t<1,\,0<u<2t^2,\,u<2(1-t)^2,\,1-2s=0
\Bigr\} \\
&\cup \left\{ (s,t,u)\in \R^3\,
\begin{array}{|c}
0<s<1,\,0<u<1/2,\,(1-u)^2-su>0,\\
(1-u)^2-(1-s)u>0,\,1-2t+tu-su=0
\end{array}
\right\} .
\end{split}
\]
The intersection of the two irreducible components is the set
\[ \Bigl\{ X \left( \tfrac{1}{2},\tfrac{1}{2},u\right)
\Bigm| 0<u<\tfrac{1}{2} \Bigr\} \cong
\Bigl\{ (s,t,u)\in \R^3 \Bigm| s=t=\tfrac{1}{2},
\, 0<u<\tfrac{1}{2} \Bigr\}. \]
The proof is also similar to that of Proposition \ref{prom}
and omitted.

\vspace{0.5\baselineskip}
Figure \ref{chio} shows two realizations of $\mat^0$.
On the left, it shows the affine image
of $X\bigl( \frac{1}{2},\frac{3}{8},\frac{1}{4} \bigr)$,
on the irreducible component $1-2s=0$.
On the right, the image of
$X\bigl( \frac{3}{4},\frac{11}{24},\frac{2}{7}\bigr)$,
so it is on the other component $1-2t+tu-su=0$.
They can be deformed continuously to each other via 
$X\left( \frac{1}{2},\frac{1}{2},u\right) \,(0<u<\frac{1}{2})$.

\end{document}